\documentclass[10pt]{amsart}

\usepackage{enumerate}
\usepackage{amsmath,amssymb}

\newtheorem{theorem}{Theorem}[section]
\newtheorem{corollary}[theorem]{Corollary}
\newtheorem{lemma}[theorem]{Lemma}

\newtheorem{proposition}[theorem]{Proposition}
\newtheorem{question}[theorem]{Question}

\theoremstyle{definition}
\newtheorem{remark}[theorem]{Remark}
\newtheorem{example}[theorem]{Example}

\def\C{{\mathbb{C}}}
\def\N{{\mathbb{N}}}
\def\R{{\mathbb{R}}}

\def\P{{\mathcal{P}}}
\def\U{{\mathcal{U}}}

\def\CL{{\mathcal{CL}}}
\def\RL{{\mathcal{RL}}}
\def\TL{{\mathcal{TL}}}

\begin{document}

\title[2-local isometries]{On 2-local nonlinear surjective isometries\\
on normed spaces and C$^*$-algebras}

\author[M. Mori]{Michiya Mori}

\address{Graduate School of Mathematical Sciences, the University of Tokyo, Komaba, Tokyo, 153-8914, Japan.}
\email{mmori@ms.u-tokyo.ac.jp}

%\thanks{}


\date{}

\begin{abstract}
We prove that, if the closed unit ball of a normed space $X$ has sufficiently many  extreme points, then every mapping $\Phi$ from $X$ into itself with the following property is affine: For any pair of points in $X$, there exists a (not necessarily linear) surjective isometry on $X$ that coincides with $\Phi$ at the two points. 
We also consider surjectivity of such a mapping in some special cases including C$^*$-algebras. 
\end{abstract}

\maketitle
\thispagestyle{empty}

\section{Introduction}
Let $X$ and $Y$ be normed spaces. 
Let $\TL(X, Y)$ denote the collection of (not necessarily linear) mappings $\Phi\colon X\to Y$ with the following property: For any $a, b\in X$, there exists a (not necessarily linear) surjective isometry $\Phi_{a, b}\colon X\to Y$ such that $\Phi_{a, b}(a)=\Phi(a)$ and $\Phi_{a, b}(b) = \Phi(b)$. 
It is clear that every $\Phi\in \TL(X, Y)$ is an isometry. 

Motivation to consider such mappings comes from the study of \emph{local mappings}. 
Let $S$ be a collection of mappings from $X$ into $Y$. 
A mapping $\Phi\colon X\to Y$ is said to be \emph{locally in $S$} if for any $a\in X$, there exists a mapping $\Phi_{a}\in S$ such that $\Phi_{a}(a) = \Phi(a)$. 
In \cite{K}, Kadison proved that every bounded linear mapping from a von Neumann algebra $M$ into a dual $M$-bimodule that is locally in the collection of derivations is a derivation. 
See also \cite{LS}. 
This result led to further research on local mappings of operator algebras. 
\v{S}emrl (essentially) introduced the following concept in \cite{S}: 
A mapping $\Phi\colon X\to Y$ is said to be \emph{2-locally in $S$} if for any $a, b\in X$, there exists a mapping $\Phi_{a, b}\in S$ such that $\Phi_{a, b}(a) = \Phi(a)$ and $\Phi_{a, b}(b) = \Phi(b)$. 
Hence a mapping in $\TL(X, Y)$ is 2-locally in the collection of surjective isometries. 
See \cite[Chapter 3]{M} and \cite[Chapter 8]{C} for more information on the study of local mappings in the setting of operator algebras. 

Let us get back to our setting. 
In this note, we are interested in properties of mappings in $\TL(X, Y)$. 
The Mazur--Ulam theorem assures that every surjective isometry between two normed spaces is affine (i.e.\ a translation of a real linear mapping).
Hence, if $\TL(X, Y)\neq \emptyset$, then it follows that $X$ and $Y$ are isometrically isomorphic, so we can identify $Y$ with $X$ as real normed spaces. 
In the rest of this note, we only consider the case $X=Y$. 
We define $\TL(X):= \TL(X, X)$.

The study of $\TL(X)$ was initiated by Moln\'ar \cite{M2}.
He proved that every mapping in $\TL(B(H))$ is surjective, where $B(H)$ denotes the collection of bounded linear operators on a separable complex Hilbert space $H$. 
He asked whether every mapping in $\TL(A)$ on the unital commutative C$^*$-algebra $A=C(K)$ for a first countable compact Hausdorff space (in particular when $K$ is the unit interval $[0,1]$ of the real line) is surjective. 
Attracted by Moln\'ar's research, Hatori and Oi proved that every mapping in $\TL(B)$ is surjective for a certain function space $B$ \cite{HO}. 
Recently, Oi gave a positive answer to Moln\'ar's question \cite{O}.

We want to consider the following question: 
Under what condition on $X$ every mapping in $\TL(X)$ surjective? 
However, it seems to be difficult to answer this question because we can easily construct bad examples (see the following sections). 
Hence we will consider a variant of the above question. 
If a mapping $\Phi\in \TL(X)$ is surjective, then $\Phi$ is affine by the Mazur--Ulam theorem. 
Therefore, it is natural to ask: 
Under what condition on $X$ every mapping in $\TL(X)$ affine? 

In this note, we first show that every mapping in $\TL(X)$ is affine if $X$ satisfies  some condition on extreme points in the closed unit ball of $X$. 
For such an $X$, we explain that the problem of $\TL(X)$ can be translated into that of local isometries. 
Using it, we give an alternate and independent proof of a positive answer to Moln\'ar's question. 
In addition, we consider $\TL(A)$ for a general unital C$^*$-algebra $A$. 
This note contains several remarks, examples and further questions. 

\section{Automatic affinity of mappings in $\TL$}
For a normed space $X$, we consider the following property: 
\begin{center}
\underline{P}\quad Every mapping in $\TL(X)$ is affine.
\end{center}
This property is important because if $X$ satisfies this, then the problem of $\TL(X)$ reduces to that of so-called local isometries, which we describe in the following paragraphs.

Let $X$ be a normed space. 
Let $\RL(X)$ (resp.\ $\CL(X)$) denote the collection of real (resp.\ complex) linear mappings $\Phi\colon X\to X$ with the property that, for every $a\in X$, there exists a real (resp.\ complex) linear surjective isometry $\Phi_a\colon X\to X$ such that $\Phi_a(a)=\Phi(a)$. 

\begin{remark}
A mapping in $\TL(X)$ with the property that the $\Phi_{a, b}$ as in the definition can be taken as a \emph{linear} mapping is sometimes called a \emph{2-local isometry}. 
A mapping in $\RL(X)$ or $\CL(X)$ is often called a \emph{local isometry}. 
However, we try to avoid using terminology of this kind in order to get rid of the possibility of confusion.  
\end{remark}

\begin{lemma}
Let $X$ be a normed space. 
\begin{enumerate}
\item Suppose $\Phi\in \TL(X)$ is an affine mapping. Then the mapping $\tilde{\Phi}\colon X\to X$ defined by $\tilde{\Phi} (x) = \Phi(x)-\Phi(0)$ satisfies $\tilde{\Phi}\in \RL(X)$. 
\item $\RL(X)\subset \TL(X)$. 
\end{enumerate}
\end{lemma}
\begin{proof}
$(1)$ It is clear that $\tilde{\Phi}\in \TL(X)$. 
By our assumption, $\tilde{\Phi}$ is real linear. 
Let $a\in X$.
Since  $\tilde{\Phi}\in \TL(X)$, there exists a surjective isometry $\tilde{\Phi}_{0, a}\colon X\to X$ such that $\tilde{\Phi}_{0, a}(0) =\tilde{\Phi}(0) = 0$ and $\tilde{\Phi}_{0, a} (a)= \tilde{\Phi}(a)$. 
By the Mazur--Ulam theorem, $\tilde{\Phi}_{0, a}$ is a real linear surjective isometry and hence $\tilde{\Phi} \in \RL(X)$. 

$(2)$
Let $\Phi\in \RL(X)$ and $a, b\in X$. 
Then there exists a real linear surjective isometry $\Phi_{b-a}\colon X\to X$ such that $\Phi_{b-a} (b-a) = \Phi(b-a)$. 
Note that 
%$\Phi_{b-a} (b-a) = \Phi_{b-a} (b)-\Phi_{b-a}(a)$ and 
$\Phi(b-a)=\Phi(b)-\Phi(a)$. 
Hence the mapping $\Phi_{a, b}\colon X\to X$ defined by $\Phi_{a, b}(x) = \Phi_{b-a}(x-a) + \Phi(a)$, $x\in X$, is a surjective isometry satisfying $\Phi_{a, b}(a) = \Phi(a)$ and $\Phi_{a, b}(b) = \Phi(b)$. 
\end{proof}

\begin{corollary}\label{cor1}
Let $X$ be a normed space. 
Suppose that every mapping in $\TL(X)$ is affine. 
Then the following two conditions are equivalent:
\begin{enumerate}
\item Every mapping in $\TL(X)$ is surjective. 
\item Every mapping in $\RL(X)$ is surjective.  
\end{enumerate}
\end{corollary}

However, in the general setting, there exists a normed space without the property  \underline{P}.
In fact, 
\begin{example}
It is well-known that the real Banach space $L^1[0, 1]$ of real-valued $L^1$-functions is almost transitive. 
Namely, for  any $a, b \in L^1[0, 1]$ with $\lVert a\rVert = \lVert b\rVert$ and any $\epsilon>0$, there exists a linear surjective isometry $T$ on $L^1[0, 1]$ such that $\lVert T(a)-b \rVert <\epsilon$.  
Let $X= (L^1[0, 1])^{\U}$ be the ultrapower of  $L^1[0, 1]$ with respect to some free ultrafilter $\U$ on $\N$. 
It is easy to see that $X$ is transitive. Namely, for  any $a, b \in X$ with $\lVert a\rVert = \lVert b\rVert$, there exists a linear surjective isometry $T$ on $X$ such that $T(a)=b$. 
It follows that every (not necessarily surjective) isometry on $X$ is in $\TL(X)$. 

Since $L^1[0, 1]$ is isometrically isomorphic to $L^1[0, 1] \oplus_1 L^1[0, 1]$, we see that $X$ is isometrically isomorphic to $X\oplus_1 X$. 
Moreover, it is known that $X$ is isometrically isomorphic to $L^1(\mu)$ for some measure $\mu$. 

We define a mapping $\Phi\colon L^1(\mu)\to L^1(\mu)\oplus_1 L^1(\mu)$ by $\Phi(f) := f_+ \oplus f_-$, where $f_+ := f\chi_{\{f>0\}}$ and $f_- := -f\chi_{\{f<0\}}$. 
Then $\Phi$ is an isometry. 
Indeed, let $f, g\in L^1(\mu)$. 
Then 
\[
\begin{split}
&\quad \lVert f-g\rVert_1 = \int \lvert f-g\rvert d\mu\\
&= \int_{\{f>0, g>0\}} \lvert f-g\rvert d\mu + \int_{\{f>0, g\leq0\}} \lvert f-g\rvert d\mu\\
&\quad\quad+ \int_{\{f\leq0, g>0\}} \lvert f-g\rvert d\mu + \int_{\{f\leq0, g\leq0\}} \lvert f-g\rvert d\mu\\
&=\int_{\{f>0, g>0\}} \lvert f_+-g_+\rvert d\mu + \int_{\{f>0, g\leq0\}}(f_+ + g_-)d\mu\\
&\quad\quad+ \int_{\{f\leq0, g>0\}}(f_- + g_+) d\mu + \int_{\{f\leq0, g\leq0\}} \lvert f_--g_-\rvert d\mu\\
&=\left(\int_{\{f>0, g>0\}} \lvert f_+-g_+\rvert d\mu + \int_{\{f>0, g\leq0\}}f_+d\mu + \int_{\{f\leq0, g>0\}}g_+ d\mu\right)\\
&\quad\quad + \left(\int_{\{f>0, g\leq0\}}g_-d\mu + \int_{\{f\leq0, g>0\}}f_- d\mu + \int_{\{f\leq0, g\leq0\}} \lvert f_--g_-\rvert d\mu\right)\\
&=\int\lvert f_+-g_+\rvert d\mu + \int \lvert f_--g_-\rvert d\mu = \lVert\Phi(f)-\Phi(g)\rVert_1.
\end{split}
\]
It is easy to see that $\Phi$ is not affine. 
We conclude that there exists a mapping in $\TL(X)$ that is not affine.
\end{example}

Nevertheless, we can show that many normed spaces satisfy the property \underline{P}. 

\begin{theorem}\label{thm1}
Let $X$ be a normed space with the property that the real linear span of extreme points in the closed unit ball is dense in $X$. 
Then every mapping in $\TL(X)$ is affine. 
\end{theorem}
\begin{proof}
Let $\Phi\in \TL(X)$. 
We show that $\Phi$ is affine.
We may assume $\Phi(0)=0$. 
Thus we need to show that $\Phi$ is real linear. 
Let $E$ denote the collection of extreme points in the closed unit ball of $X$.  

First we prove the following: 
For any $x, y\in X$ and $\lambda\in \R$, we have $y- x\in \lambda E$ if and only if $\Phi(y)-\Phi(x)\in \lambda E$. 
Moreover, for any $x\in E$ and $\lambda\in \R$, we have $\Phi(\lambda x) = \lambda \Phi(x)$. 
Indeed, since $\Phi$ is in $\TL(X)$, for $x, y\in X$, there exists a surjective isometry $\Phi_{x, y}\colon X\to X$ such that $\Phi_{x, y} (x)=\Phi(x)$ and $\Phi_{x, y}(y) = \Phi(y)$. 
Let $\lambda\in \R$. 
It follows by the Mazur--Ulam theorem that $\Phi_{x, y}$ is affine and thus $y- x\in \lambda E$ if and only if $\Phi(y)-\Phi(x)\in \lambda E$. 
Let $x$ be in $E$ and $\lambda<0$ be a negative real number. 
Since $x-\lambda x \in (1+|\lambda|)E$, we have $\Phi(x) -\Phi(\lambda x)\in (1+|\lambda|) E$. 
By the equation
\[
\frac{\Phi(x)-\Phi(\lambda x)}{ 1+|\lambda|} = 
\frac{1}{ 1+|\lambda|} \cdot \Phi(x)+ \frac{|\lambda|}{1+|\lambda|}\cdot \frac{\Phi(\lambda x)}{\lambda}
\]
we obtain $\Phi(\lambda x) = \lambda\Phi(x)$. 
If $\lambda$ is a positive real number, then we have $\Phi(\lambda x) = -\lambda \Phi(-x) = \lambda \Phi(x)$, hence we obtain the claim. 

Next we prove: 
If $x\in E$, $y\in X$ and $\lambda \in \R$ , then $\Phi(y+\lambda x) = \Phi(y) + \lambda\Phi(x)$. 
By the preceding paragraph, we have $\Phi(\lambda x) = \lambda \Phi(x)$. 
Consider another mapping in $\TL(X)$ defined by $X\ni a\mapsto \Phi(y+a) -\Phi(y)\in X$. 
Then the same discussion shows that there exists an element $z\in E$ such that $\Phi(y+ \lambda x) -\Phi(y) = \lambda z$ for any $\lambda\in \R$. 
It suffices to show that $z=\Phi(x)$. 
Since $\Phi$ is an isometry, we have for any $\lambda\in \R$
\[
\begin{split}
\lVert y\rVert &= \lVert (y+\lambda x)-\lambda x\rVert\\
&=\lVert \Phi(y+\lambda x) -\Phi(\lambda x)\rVert = \lVert (\Phi(y) + \lambda z) - \lambda \Phi(x)\rVert = \lVert\Phi(y) + \lambda(z-\Phi(x))\rVert. 
\end{split}
\]
Consider the limit $\lambda\to\infty$ to obtain $z= \Phi(x)$. 

By what we have proved, we obtain $\Phi(\sum_{1\leq k\leq n} \lambda_k x_k) = \sum_{1\leq k\leq n} \lambda_k\Phi(x_k)$ for any $n\in \N$ and $\lambda_k\in \R$, $x_k\in E$, $1\leq k\leq n$. 
It follows that $\Phi$ is real linear on $\operatorname{span}(E)$. 
Since $\Phi$ is an isometry and $\operatorname{span}(E)$ is dense in $X$, $\Phi$ is real linear on $X$. 
\end{proof}

\begin{question}
Is every mapping in $\TL(X)$ on a separable Banach space $X$ affine? 
\end{question}

\section{Non-surjective examples}
We give several examples of non-surjective mappings in $\CL(X)$ (which is contained in $\TL(X)$). 

\begin{example}
Let $H$ be an infinite dimensional complex Hilbert space. 
Then there exists a non-surjective mapping in $\CL(H)$. 
In fact, $H$ is isometrically isomorphic to $H\oplus_2 H$, and the mapping $\Phi\colon H\to H\oplus_2 H \,(\cong H)$ defined by $\xi\mapsto \xi\oplus 0$ is such an example. 
For another (arbitrary) normed space $X$, consider the $\ell_p$-direct sum $H\oplus_p X$ ($1\leq p\leq \infty$). 
Then the mapping $\Phi\oplus \operatorname{id}_X \colon H\oplus_p X\to H\oplus_p X$ is a non-surjective mapping in $\CL(H\oplus_p X)$. 
\end{example}

\begin{example}\label{ex}
Nonseparable Banach spaces give other examples. 
Let $I$ be an uncountable set and $1\leq p<\infty$ be a real number. 
Consider the space $\ell_p(I)$ (resp.\ $c_0(I)$) of complex-valued functions $f$ on $I$ with $\lVert f\rVert_p^p := \sum_{i\in I} \lvert f(i)\rvert^p <\infty$ (resp.\ with $\lvert f(i)\rvert >\epsilon$ for finitely many $i\in I$ for any $\epsilon>0$, and $\lVert f\rVert :=\sup_{i\in I} \lvert f(i)\rvert$). 
Fix an element $i_0\in I$ and take a bijection $\tau\colon  I\setminus \{i_0\} \to I$. 
Then the mapping $\Phi\colon \ell_p(I)\to \ell_p(I)$  (resp.\ $\Phi\colon c_0(I)\to c_0(I)$) defined by $(\Phi(f))(i) := f(\tau(i))$, $i\neq i_0$ and $(\Phi(f))(i_0):= 0$ is a non-surjective mapping in $\CL(\ell_p(I))$ (resp.\ $\CL(c_0(I))$).  
Consider the unitization $c(I)$ of the commutative C$^*$-algebra $c_0(I)$. 
Then we can find an example of non-surjective unital complex linear mapping in $\CL(c(I))$ in a similar way. 
For a nonseparable Hilbert space $H$, consider the space $K(H)$ of compact operators on $H$. 
A similar construction as above gives examples of non-surjective mappings in $\CL(K(H))$ and $\CL(K(H) + \C \operatorname{id}_H)$, too.
\end{example}

\begin{example}
We can also give many examples in the case of separable normed spaces. 
Consider the space $c_{00}$ of all finitely supported sequences of complex numbers. 
Let $1\leq p\leq\infty$. 
Endow $c_{00}$ with the $\ell_p$-norm. 
Then the mapping $(a_n)_{n\geq 1} \mapsto (a_{n-1})_{n\geq 1}$, $(a_n)_{n\geq 1}\in c_{00}$, $a_0:= 0$, is an example for $(c_{00}, \lVert\cdot \rVert_p)$. 
\end{example}

Note that we can give similar examples in $\RL$ for corresponding real normed spaces.

\section{$\TL$ of C$^*$-algebras}
In this section, motivated by Moln\'ar's work \cite{M2}, we consider mappings in $\TL(A)$ for a C$^*$-algebra $A$. 
See for example \cite{KR} for the basic properties of operator algebras, but remark that every C$^*$-algebra in \cite{KR} is assumed unital.   

Let $A$ be a C$^*$-algebra. 
Let $A_{sa} := \{a\in A\mid a=a^*\}$ and $\P(A):= \{p\in A\mid p=p^*=p^2\}$ denote the self-adjoint part of $A$ and the collection of projections in $A$, respectively. 
When $A$ is unital, let $U(A):= \{u\in A\mid uu^*=1=u^*u\}$ denote the unitary group of $A$. 
It is well-known that the connected component of $1$ in $U(A)$ is equal to $U_0(A) := \{e^{ia_1}e^{ia_2}\cdots e^{ia_n} \mid n\in \mathbb{N},\,\,a_1, a_2,\ldots, a_n\in A_{sa}\}$.
A \emph{Jordan $^*$-homomorphism} $\Phi$ from $A$ into another C$^*$-algebra is a complex linear mapping such that $\Phi(x^*) = \Phi(x)^*$ and $\Phi(xy+yx)=\Phi(x)\Phi(y)+\Phi(y)\Phi(x)$ for any $x, y\in A$. 

The following is well-known. 

\begin{lemma}[a version of the Russo--Dye theorem, see for example {\cite[Exercise 10.5.4]{KR}}]\label{lem1}
Let $A$ be a unital C$^*$-algebra. 
Then the closed convex hull of $U_0(A)$ is equal to the closed unit ball of $A$. 
\end{lemma}

We can easily show that every unitary in $A$ is an extreme point of the closed unit ball of $A$. 
Thus by Theorem \ref{thm1} and Corollary \ref{cor1}, we have:
\begin{corollary}\label{cor2}
Let $A$ be a unital C$^*$-algebra. 
Then the following conditions are equivalent. 
\begin{enumerate}
\item Every mapping in $\TL(A)$ is surjective.
\item Every mapping in $\RL(A)$ is surjective. 
\end{enumerate}
\end{corollary}

The following are also known.

\begin{lemma}[See {\cite[Corollary 3.3]{D}}]\label{lem2}
Let $A$ and $B$ be unital C$^*$-algebras and suppose $\Phi\colon A\to B$ is a real linear  surjective isometry. 
Then $\Phi(1)\in U(B)$.  
Moreover, there exists a Jordan $^*$-isomorphism $J\colon A\to B$ and a central projection $p\in \P(B)$ such that $\Phi(x)= \Phi(1)(J(x)p +J(x)^*p^{\perp})$, $x\in A$, where $p^{\perp}$ means $1-p$. 
\end{lemma}
\begin{corollary}\label{cor}
Let $A$ and $B$ be unital C$^*$-algebras and suppose $\Phi\colon A\to B$ is a real linear  surjective isometry. Then we have the following:
\begin{enumerate}
\item $\Phi(U(A)) =U(B)$. 
\item If $Q\in \P(A)$ and $\Phi(Q)\in B_{sa}$, then the spectrum $\sigma(\Phi(Q))$ of $\Phi(Q)$ is a subset of $\{1, 0, -1\}$. 
\end{enumerate}
\end{corollary}
\begin{proof}
Take $J$ and $p$ as in the above lemma.\\
$(1)$ Clear since $J$ maps a unitary to a unitary.\\ 
$(2)$ Since $J$ maps a projection to a projection, $P:=J(Q)p+J(Q)p^{\perp}$ is a projection in $B$. 
It follows that $\Phi(Q)^2 = \Phi(Q)^*\Phi(Q)= P^*\Phi(1)^*\Phi(1)P = P$. 
Since $\Phi(Q)\in B_{sa}$, this equation implies that $\sigma(\Phi(Q))\subset \{1, 0, -1\}$. 
\end{proof}

\begin{lemma}[Russo--Dye, {\cite[Corollary 2]{RD}}]\label{lem3}
Let $A$ and $B$ be unital C$^*$-algebras. 
If $\Phi\colon A\to B$ is a complex linear mapping with $\Phi(1) = 1$ and $\Phi(U(A))\subset U(B)$, then $\Phi$ is a Jordan $^*$-homomorphism.  
\end{lemma}

We give a proof of an answer to Moln\'ar's question. 
See also \cite{O} by Oi for another proof. 

\begin{theorem}\label{thm2}
Let $A$ be a unital commutative separable C$^*$-algebra, or more generally, let $A=C(K)$ for a first countable compact Hausdorff space $K$. 
Then every mapping in $\TL(A)$ is surjective.
\end{theorem}
\begin{proof}
By Corollary \ref{cor2}, it suffices to show that every $\Phi\in \RL(A)$ is surjective. 
By Corollary \ref{cor}, $\Phi(1)\in U(A)$. 
Considering $\Phi(1)^*\Phi(\cdot)$ instead of $\Phi$, we may assume $\Phi(1) = 1$. 
We have $\lVert \Phi(i) \pm 1\rVert = \lVert \Phi(i) - \Phi(\mp 1)\rVert = \lVert i\pm 1\rVert = \sqrt{2}$. 
By Corollary \ref{cor}, $\Phi(i)\in U(A)$. 
These facts assure that $\sigma(\Phi(i))\subset \{i, -i\}$. 
Thus there exists a projection $p$ in $A$ such that $\Phi(i) = ip -ip^{\perp}$, where $p^{\perp}:= 1-p$.  

Since $A$ is a commutative C$^*$-algebra, we can decompose $A$ into a direct sum: $A= Ap \oplus_{\infty} Ap^{\perp}$. 
Considering the mapping $A\ni x\mapsto \Phi(x)p + \Phi(x)^*p^{\perp}\in A$ (which is also a mapping in $\RL(A)$) instead of $\Phi$, we may in addition assume $\Phi(i) = i$. 

Consider the subset $\tilde{U}:=\{u\in U_0(A)\mid \Phi(iu)=i\Phi(u)\}$ of $U_0(A)$, which contains $1$. 
The same discussion as in the first paragraph of this proof assures that $\sigma(\Phi(u)^*\Phi(iu)) \subset \{i, -i\}$ for any $u\in U(A)$. 
By the continuity of $\Phi$, it follows that $\tilde{U}$ is open and closed in the connected set $U_0(A)$. 
Thus we obtain $\tilde{U} = U_0(A)$. 
Lemma \ref{lem1} and the real linearity of $\Phi$ imply that $\Phi$ is complex linear. 
By Lemma \ref{lem3}, $\Phi$ is a Jordan $^*$-endomorphism. 
Since $A$ is commutative, $\Phi$ is a $^*$-endomorphism. 

By the assumption that $A= C(K)$ for some first countable compact Hausdorff  space $K$, the rest of the proof can be completed by almost the same way as in the proof of \cite[Lemma]{MZ}. 
Indeed, we see that $\Phi$ is of the form $f\mapsto f\circ \varphi_{0}$, $f\in C(K)$, for some continuous surjection $\varphi_0\colon K\to K$. 
If $\varphi_0$ is not injective, then there exists a point $z\in K$ such that $\varphi_0^{-1}(z)$ has more than one point. 
We can take a nonnegative function $f\in C(K)$ with $f^{-1}(1) = \{z\}$ and $\lVert f\rVert = 1$. 
It follows that $(\Phi(f))^{-1}(1)$ has more than one point, but this contradicts $\Phi\in \RL(C(K))$ by the classical Banach--Stone theorem combined with Lemma \ref{lem2}. 
\end{proof}

Note that the weaker result that every mapping in $\CL(A)$ is surjective for the same $A$ was proved in \cite[Lemma]{MZ}. 
Example \ref{ex} implies that we cannot drop the assumption on first countability. 
Theorem \ref{thm2} gives a positive answer to the following question in the commutative setting. 

\begin{question}
Let $A$ be a unital separable C$^*$-algebra. 
Is every mapping in $\TL(A)$ (or equivalently, in $\RL(A)$) surjective?
\end{question}

However, the author could not obtain an answer in the general noncommutative case. 
Recall that Moln\'ar's result in \cite{M2} means that the von Neumann algebra $B(H)$ (which is nonseparable as a C$^*$-algebra) has the property that every mapping in $\TL(B(H))$ is surjective.   
The following proposition may give an idea for tackling the above question. 

\begin{proposition}\label{vna}
Let $A$ be a (not necessarily separable) unital C$^*$-algebra with the property that the real linear span of $\P(A)$ is dense in $A_{sa}$. 
Suppose $\Phi$ is a mapping in $\RL(A)$ with $\Phi(1)=1$. 
Then the following holds: 
There exists a projection $p\in \P(A)$ such that $\Phi(i) = ip-ip^{\perp}$. 
This projection commutes with every element in $\Phi(A)$. 
Define $\tilde{\Phi}\colon A\to A$ by $\tilde{\Phi}(x):= \Phi(x)p +\Phi(x)^*p^{\perp}$. 
Then $\tilde{\Phi}$ is a unital Jordan $^*$-endomorphism.  
\end{proposition}

\begin{remark}
It is clear that every von Neumann algebra has the property as in the assumption of this proposition. 
It is known that there are many other C$^*$-algebras with this property (cf.\ \cite{P}). 
\end{remark}

\begin{proof}[Proof of Proposition \ref{vna}]
First we show that $\Phi(iA_{sa})\subset iA_{sa}$. 
It suffices to show $\Phi(ia)\in iA_{sa}$ for any $a\in A_{sa}$ with $\lVert a\rVert = 1$. 
Decompose $\Phi(ia)$ into a sum $a_1+ ia_2 $ with $a_1, a_2\in A_{sa}$. 
Since $\Phi\in \RL(A)$ and $\Phi(1)=1$, for any $t\in \R$, we have 
\[
\sqrt{1+t^2} = \lVert 1-iat\rVert = \lVert \Phi(1)-\Phi(iat)\rVert = \lVert 1-a_1 t -  ia_2 t\rVert \geq \lVert 1-a_1 t\rVert. 
\]
Take the positive and negative parts of $a_1$: $a_1=a_{1+} -a_{1-}$, $a_{1+}, a_{1-}\in A_+$, $a_{1+}a_{1-} =0$. 
We have $\lVert 1-a_1 t\rVert\geq 1-t\lVert a_{1+}\rVert$ for any $t<0$, and $\lVert 1-a_1 t\rVert\geq 1+t \lVert a_{1-}\rVert$ for any $t>0$. 
By the above inequality, we obtain $a_{1+}=0=a_{1-}$. 
Hence $\Phi(ia) = a_1+ ia_2  =  ia_2  \in iA_{sa}$. 

Define a mapping $\Phi'\colon A_{sa}\to A_{sa}$ by $\Phi'(a) := -i\Phi(ia)$. 
By Corollary \ref{cor}, we have $\Phi(U(A))\subset U(A)$. 
Hence for any $Q\in \P(A)$, there exists a projection $P\in \P(A)$ such that $\Phi'(Q-Q^{\perp}) = P-P^{\perp}$. 
In particular, we have $(-i\Phi(i) = )\, \Phi'(1) = p - p^{\perp}$ for some $p\in \P(A)$. 
We show that $p$ commutes with any $P$ as above. 
Since $\Phi'$ is linear, we have $-i\Phi(iQ) =\Phi'(Q) = \Phi'((1+Q-Q^{\perp})/2) = (p-p^{\perp} + P-P^{\perp})/2 = p- P^{\perp}$. 
Since $\Phi\in \RL(A)$, there exists a real linear surjective isometry $\Phi_{iQ}\colon A\to A$ with $\Phi_{iQ}(iQ) = \Phi(iQ) = i(p-P^{\perp})$. 
Observe that the mapping defined by $x\mapsto -i\Phi_{iQ}(ix)$ is also a real linear surjective isometry on $A$.
Therefore, by Corollary \ref{cor}, the spectrum $\sigma(p-P^{\perp})$ must be a subset of $\{1, 0, -1\}$, and hence $p$ and $P$ commute. 
It follows that $p$ commutes with $\Phi'(Q)$ for any $Q\in \P(A)$. 
Recall the assumption that $\P(A)$ is total in $A_{sa}$. 
Since $\Phi'$ is an isometry, $p$ commutes with every element in $\Phi'(A_{sa})=-i\Phi(iA_{sa})$. 

Consider another mapping $\Phi''\colon A\to A$ defined by $\Phi''(x) := -i(p-p^{\perp}) \Phi(ix)$. 
Then $\Phi''\in \RL(A)$ and $\Phi(1) = 1$. 
The same discussion as in the preceding paragraph shows that $( i(p-p^{\perp}) \Phi(A_{sa}) = )\, \Phi''(iA_{sa}) \subset iA_{sa}$ and that 
every element in $-i \Phi''(iA_{sa}) = (p-p^{\perp}) \Phi(A_{sa})$ commutes with $-i \Phi''(i) = p-p^{\perp}$. 
It follows that every element in $\Phi(A_{sa})$ commutes with $p$. 
Therefore, every element in $\Phi(A)$ commutes with $p$. 

Now we can define the real linear isomerty $\tilde{\Phi}\colon A\to A$ by $\tilde{\Phi}(x):= \Phi(x)p +\Phi(x)^*p^{\perp}$. 
We have $\tilde{\Phi}(1) = 1$, $\tilde{\Phi}(i) = i$ and $\tilde{\Phi}(U(A))\subset U(A)$. 
Put $U_1 := \{u\in U_0(A)\mid \tilde{\Phi}(iu) = i\tilde{\Phi}(u)\}$. 
Then $1\in U_1$. 
The same discussion as in the first half of this proof implies that $\sigma(\tilde{\Phi}(u)^*\tilde{\Phi}(iu)) \subset \{i, -i\}$ for any $u\in U(A)$. 
Hence the continuity of $\tilde{\Phi}$ shows that $U_1$ is open and closed in the connected set $U_0(A)$, thus $U_1= U_0(A)$.  
Since $\tilde{\Phi}$ is an isometry, Lemma \ref{lem1} implies that $\tilde{\Phi}$ is complex linear, and hence $\tilde{\Phi}$ is a Jordan $^*$-endomorphism by Lemma \ref{lem3}. 
\end{proof}

\begin{question}
Can we drop the assumption on $\P(A)$ in Proposition \ref{vna}? 
(Note that the first paragraph of this proof is valid for any unital C$^*$-algebra.)
\end{question}

The author does not know the answer to the following question on \emph{local $^*$-automorphisms}, which are closely related to local isometries. 
\begin{question}\label{la}
Let $A$ be a unital separable C$^*$-algebra. 
Suppose $\Phi\colon A\to A$ is a $^*$-endomorphism with the following property: 
For every $a\in A$, there exists a $^*$-automorphism $\Phi_{a}\colon A\to A$ such that $\Phi_a(a) = \Phi(a)$. 
Is such a $\Phi$ always surjective? 
\end{question} 
Example \ref{ex} implies that this does not necessarily hold in nonseparable cases.  
The answer to Question \ref{la} is YES when $A$ is singly generated as a C$^*$-algebra. 
If there exists a counterexample to the statement of this question, then it is also a non-surjective example in $\CL(A)$ (and thus in $\RL(A)$ and $\TL(A)$).\medskip\medskip 

\textbf{Acknowledgements} \quad 
The author appreciates Yasuyuki Kawahigashi, the advisor of the author, for invaluable support. 
The author is also grateful to Gy\"{o}rgy P\'al Geh\'er, Osamu Hatori, Lajos Moln\'ar and Gilles Pisier for helpful advices on a draft of this note. 
This work began during the author's visit to the Research Institute for Mathematical Sciences, a Joint Usage/Research Center located in Kyoto University, where Hatori gave a talk on 2-local isometries. 
This work was supported by Leading Graduate Course for Frontiers of Mathematical Sciences and Physics (FMSP) and JSPS Research Fellowship for Young Scientists (KAKENHI Grant Number 19J14689), MEXT, Japan.

\end{document}